\author{Grzegorz Tomkowicz}

\documentclass[12pt, leqno]{article}

\usepackage{amsmath,amsthm}

\usepackage{amssymb,latexsym}

\usepackage{enumerate}
\usepackage[T1]{fontenc}

\newtheorem{thm}{Theorem}[section]

\newtheorem{cor}[thm]{Corollary}

\newtheorem{lem}[thm]{Lemma}

\theoremstyle{definition}

\newtheorem{defin}[thm]{Definition}

\begin{document}

\baselineskip=17pt

\title{On some geometrization of compact metric spaces: A solution to the Banach-Ulam conjecture}
\maketitle

\renewcommand{\thefootnote}{}



\renewcommand{\thefootnote}{\arabic{footnote}}

\setcounter{footnote}{0}

\begin{abstract} We propose a geometrization of compact metric spaces that is based on ideas of S. Banach and J. Mycielski. Then we prove the following conjecture of S. Banach and S. Ulam from 1935: in every compact metric space there exists a finitely additive probability measure, invariant under congruences. Moreover, our techniques allow us to solve a problem of M. Talagrand related to the Marczewski problem and the Banach-Tarski paradox with pieces having the property of Baire. We give also a very simple proof of the conjecture of Ulam about the
product Lebesgue measure in the Hilbert cube and explain the existing results about congruence-invariant Borel measures in the language of our geometrization.       
 
\end{abstract}
 
\section{Introduction} In 1935 S. Banach and S. Ulam posed in [12] the following problem:

\vspace{1cm}
\emph{Does for any compact metric space $X$ exist a finitely additive probability Borel measure $m$ that is congruence-invariant?}
	\vspace{1cm}
    
    Recall that congruences are isometries defined on subsets of $X$ that do not necessarily extend to the whole space. The first
		 step toward solving this problem was made by Banach.
		 Using the Banach-Mazur limits, Banach provided in [2] a construction of the Haar measure, and his construction uses coverings of compact
		 subsets to define a finitely additive congruence-invariant measure on compact subsets of locally compact topological groups.
		 Later, J. Mycielski [10] generalized Banach's method to obtain a finitely additive congruence-invariant measure defined on all
		 compact subsets of any compact metric space. Mycielski  obtained also countably additive measures that, as proved in [3], 
		 are congruence-invariant in the
		 spaces where a homogeneity appears. However, there exist compact metric spaces where there is no countably additive congruence-invariant
 probability measure, e.g. countable compact metric spaces. In the case of countable compact metric spaces, the 
		 Banach-Ulam conjecture was confirmed by R. Davies and A. Ostaszewski [4].\\
		   \indent In the present paper, we solve the Banach-Ulam problem in affirmative. In our constructions we develop the idea of Banach
			 and link it with a theorem of Tarski that equates the existence of measures with the lack of paradoxical decompositions. And our idea
			is based on the \emph{geometrization} of compact metric spaces. Recall, that Klein's program is about interpreting geometry of a
			 space $X$ as a set of invariants with respect to the group acting on $X$. In the same spirit, we introduce in any compact metric space 
			 the objects that are congruence-invariant. And our invariants are based on the finitely additive measures defined on compact subsets
				constructed by Banach and Mycielski. Our geometrization program has numerous consequences and explains many results
				 about finitely additive and countably additive congruence-invariant measures.\\
				   \indent The most interesting connection between our geometrization and paradoxical decompositions and invariant measures is
				    the Marczewski problem. It asks if there exists
								a paradoxical decomposition of the Euclidean unit sphere $\mathbb{S}^{n-1}$ or the Euclidean cube $[0,1]^n$ for 
								 $n \geq 3$ with the pieces having the property of Baire. R. Dougherty and M. Foreman [5] proved that indeed such
								 a decomposition is possible. However, in the same paper (p. 123) M. Talagrand asked if the paradoxical decomposition
								 is still possible with the pieces $A_i$ such that for any closed subset $Y$ of the aforementioned spaces,
								 $Y \cap A_i$ has the property of Baire.
								   It appears that Banach-Ulam problem and Talagrand problem generalize to a very general problem about exsistence of
									finitely additive congruence-invariant measures in compact metric spaces. And our Theorem 4.2 solves the general problem,
									 in particular we show that there is no paradoxical decomposition of the $\mathbb{S}^{n-1}$ or the Euclidean 
									 cube $[0,1]^n$ with the pieces having the stronger property of Baire as demanded by Talagrand. What is interesting
									 is that this kind of pieces eliminates the \emph{telescoping property} of the pieces apearing in classical Banach-Tarski
									 decompositions. Namely, comeager parts of the pieces belong to the domains of congruences transporting the pieces to 
									 the destination points, and this phenomenon eliminates any kind of paradoxical decomposition.\\
									  \indent We provide also a link between our geometrization and the existence of countably additive 
										congruence-invariant measures (Thms. 5.1 and 5.2) that allows us to give a very concise solution the problem of Ulam
										 about the Lebesgue product measure in the Hilbert cube (Thm. 6.1).

\section{Preliminaries} In this section we will gather all the notions and known results needed to proceed with our geometrization
 program. We begin with the
 notion of Banach-Mycielski content defined first by Banach in [2] and then refined by Mycielski [10]. In order to define the 
 content consider the interval $[0, \infty]$ equipped with its natural compact topology. Moreover, let $\mathcal{F}$ be a non-principial
 ultrafilter defined on the set of non-negative integers $\mathbb{N}$. Now, for every sequence $(x_n)$ with $x_n \in [0, \infty]$
 we define $\lim_{n \to \mathcal{F}} x_n$ to be the unique number $x \in [0, \infty]$ such that every neighbourhood $U$ of $x$
 has the property that
   $\{n \in \mathbb{N}: x_n \in U \} \in \mathcal{F}$. The limit (which may be regarded as the Banach-Mazur limit, see [1], p. 34 or [9])
	 is a conservative extension of the ordinary limit and has their usual properties.\\
	  \indent We will need also the notion of \emph{entropy} $E(C,t)$ defined in any compact metric space $X$: 
		
		$$E(C,t) = min \{ card \mathcal{K}: \mathcal{K} \ \textrm{ is a covering of} \ C, diam(\mathcal{K}) < t\},$$
     where $C$ is a compact subset of $X$, $t \in (0,1]$ and $diam(\mathcal{K})$ means that all the sets in $\mathcal{K}$
		 have the diametrers less than $t$.\\
		  \indent Now, the Banach-Mycielski content $\lambda$, defined for any compact subset $K$ of any compact metric space 
			 is the following limit:
			
			  $$ \lambda(K) = \lim_{n \to \mathcal{F}} \frac{E(K, \frac{1}{n})}{E(X, \frac{1}{n})}.$$
			
				The content has the following properties described explicitely in [10], Lemma A:
				
			\begin{tabbing}
$(a)$ \= \hspace{1.5cm} \= $\lambda(X) = 1$; \\[1.5mm]
$(b)$ \> \> $\lambda(A) \leq \lambda(B)$ if $A \subset B$; \\[1.5mm]
$(c)$ \> \> $\lambda(A \cup B) \leq \lambda(A) + \lambda(B)$; \\[1.5mm]
$(d)$ \> \> $\lambda(A \cup B) = \lambda(A) + \lambda(B)$ if $A \cap B = \emptyset$; \\[1.5mm]
$(e)$ \> \> $\lambda(A) = \lambda(\sigma(A))$, where $\sigma$ is a congruence defined on $A$.
\end{tabbing}
				
				  Developing the ideas of Banach [2], Mycielski [10] constructs a countably additive probability measure $\mu$ in any compact
					metric space $X$ such that $\mu(U) = \mu(V)$ for congruent open sets. Here we construct first the auxilary function on open sets,
					$\mu_0(U) = sup \{\lambda(K) : K \subset U, K = \overline{K}\}$, then the outer measure 
					 $\mu^{*}(A) = inf \{\mu_0(U): A \subset U, U = Int U\}$. And finally we apply the Carath\' eodory theorem to get the
					 a probability Borel measure in $X$ we will call \emph{Banach-Mycielski measure}.
					 The measure in general is not congruence-invariant, e.g. countable compact metric spaces do not bear any such measure.
					
					\indent The next notion we will apply in our constructions is the Tarski's \emph{semigroup of equidecomposability types} described
					extensively in [14], Chapter 10 in the case of group actions and in [13], p.233 in the case of congruences in compact 
					metric spaces (see also remarks after Cor. 11.3 in [14]). So let $X$ be a compact metric space and let $\mathcal{C}$ be the set
					 of all congruences defined on the subsets of $X$. Define now the set $X^{*} = X \times \mathbb{N}$ and say that 
					 $(\sigma, \pi)(x,n) = (\sigma(x), \pi(x))$ is a \emph{Tarski's operation}, 
					 where $\sigma \in \mathcal{C}$ and $\pi$ is a permutation of $\mathbb{N}$. Let $A \subset X^{*}$, we will call
					those $n \in \mathbb{N}$ for which there exists at least one element of $A$ with the second coordinate $n$, \emph{the levels
					 of} $A$.\\
				    \indent We will say that two sets $A$ and $B$ of compact metric space $X$ are \emph{equidecomposable} if there exists a 
						finite partition $\{A_i\}_{i=1}^k$ of $A$ and the congruences $\sigma_1,...,\sigma_k$ such that $\{\sigma_(A_i)\}_{i=1}^k$
						is a partition of $B$. And $A^{*}, B^{*} \subset X^{*}$ are \emph{Tarski equidecomposable} if instead of $\sigma_i$ we use
						 the Tarski's operations.\\
						   \indent Consider now the subsets of $X^{*}$ that have only finitely many levels and call them \emph{bounded}, 
							then for any such set $A$ we can define
							 the equivalence class $[A]$ with respect to the Tarski equidecomposability. Denote the collection of all types 
							 of bounded sets by $\mathcal{S}$ and observe that for any $[A], [B] \in \mathcal{S}$ we can define the addition $+$ by
							 $[A] + [B] = [A \cup B']$, where $B' = \{(b, m +k): (b,m) \in B\}$, where $k$ is large enough that $A$ and
							 $B'$ are disjoint. It is easy to verify that the addition is well-defined and thus we will call the set $\mathcal{S}$
							  the \emph{type semigroup}. In $\mathcal{S}$ we can intoduce also the multiplication $n \cdot \alpha := \alpha+ ...+ \alpha$
								with $n$ summands; and the order $\leq$ by $\alpha \leq \beta$ if and only if $\alpha + \gamma = \beta$ for some 
								 $\gamma \in \mathcal{S}$.\\
								 \indent Let $X$ be a compact metric space and let $\mathcal{A}$ be a $\sigma$-algebra of subsets of $X$, that is 
								congruence-invariant, that is $\sigma(A) \in A$, whenever $A \in \mathcal{A}$ and $\sigma$ is a congruence defined on $A$.
								 Then we define the type semigroup for $\mathcal{A}$, denoted by $\mathcal{S}(\mathcal{A})$ as the set of all elements of
								 the form:
								
								 $$\mathcal{A}^{*}= \{A \subset X^{*}: \ \textrm{for some} \ n \in \mathbb{N}, A = \bigcup_{m <n} A_m \times \{m\},
								 A_m \in \mathcal{A}\}.   $$
								  
									Then we define the equivalence classes $[A]$ of elements in $\mathcal{A}^{*}$ using the Tarski equidecomposability
									 with pieces restricted to $\mathcal{A}^{*}$. Finally, we arrive to the fundamental theorem of Tarski
									 (see [14], Thm. 3.6 and Thm. 11.1 for a proof):
									
									\begin{thm} Let $X$ be a compact metric space and let $\mathcal{A}$ be a congruence-invariant $\sigma$-algebra
										 of subsets of $X$ then the following conditions are equivalent:
										 
										  $(i)$ \ \ \ \ \  There is no $n \in \mathbb{N}$ with $(n+1) \cdot [X] = n \cdot [X]$ in
											$\mathcal{S}(\mathcal{A})$;  
											
											$(ii)$ \ \ \ \ \  There is a finitely additive congruence-invariant measure\\
											  $m: \mathcal{A} \rightarrow  [0, \infty]$
											 with $m(X) = 1$.
											    \end{thm}
													
											\indent Let $X$ be a compact metric space, we will say that a set $A \subset X$ has the property of Baire if 
											it can be represented as the symmetric difference $U \triangle M$, where $U$ is an open and $M$ is a meager
											subset of $X$. It is well-known that for every set $A$ in compact metric space that has the property of Baire
											 there exists a unique regular open set $V$ such that $A = V \triangle M$ and $M$ is a meager subset of $X$.
											  We will say that a set $A \subset X$ is \emph{regular} if for every closed $F \subset X$, $F \cap A$ has the 
												property of Baire. It is easy to verfy that regular sets form a $\sigma$-algebra containing Borel sets that
												is congruence-invariant. We call this algebra \emph{regular} and denote it by $Reg(X)$. Moreover, we call
												any partition into regular sets of any Borel subset of $X$, a \emph{regular partition}.\\
												 \indent In what follows we use the standard topological notation, that is for a set $A$, we denote by
												 $IntA$ its interior, by $\overline{A}$ its closure and by $\partial{A}$ its boundary.
                         
												  \indent We will need also the following results from descriptive set theory that links our constructions
													with the Cantor-Bendixson theorem:

\begin{lem} Let $(A_{\alpha})$ be a decreasing transfinite sequence of $F_{\sigma}$ sets of a compact metric space $X$ such that
 $A_{\alpha} = \bigcup_{n=1}^{\infty} F^{\alpha}_n$, where $F^{\alpha}_{n}$ is a closed subset of $X$, 
 $F^{\alpha}_{n} \subset F^{\alpha}_{n+1}$ and $F^{\alpha+1}_{n}$ is a nowhere dense subset of $F^{\alpha}_{n}$. Moreover, for limit ordinals
  $\lambda$, it holds 
   $F_{n}^{\lambda} = \bigcap_{\beta < \lambda} F_{n}^{\beta}$.
 Then there exists a countable ordinal $\gamma$ with $A_{\gamma} = \emptyset$. \end{lem}

 \emph{Proof.} Fix a positive integer $m$ and observe that since $F^{\alpha +1}_m$ is nowhere dense in $F^{\alpha}_m$, then we 
 obtain a strictly decreasing transfinite sequence $(F^{\alpha}_m)$ of closed subsets of $X$. Therefore, there exists a countable
 ordinal $\gamma_m$ with $F^{\gamma_m}_m = \emptyset$ (see [8], Thm. 6.9). Since $A_{\alpha} = \bigcup_{n=1}^{\infty} F^{\alpha}_n$
 we obtain in this way countably many countable ordinals $\gamma_n$. Now, it is enough to take 
  $\gamma = \sup \{\gamma_n: n = 1,2,...\}$. $\Box$\\
	
	In our proof of Banach-Ulam conjecture and the Talagrand's problem we encounter a more general situation than in the case of 
	our above lemma:\\
	
	\begin{thm} Let $X$ be a compact metric space and let $(A_{\alpha})$ be a transfinite sequence of 
	subsets of $X$ such that $A_{0} = \bigcup_{n=1}^{\infty} F^{0}_n$, where $F^{0}_n$ is nowhere dense and closed, 
	 $A_{\alpha + 1} = \bigcup_{n=1}^{\infty} F^{\alpha+1}_n$, where $F^{\alpha+1}_n = \bigcup_{j=1}^{\infty} N^{n}_j$ and $N^{n}_j$ is a closed nowhere dense subset of $F^{\alpha}_n$ and
		$A_{\lambda} = \bigcap_{\beta < \lambda} A_{\beta}$ for the limit ordinals $\lambda$. Then, for every limit ordinal $\lambda$, $A_{\lambda}$ is an $F_{\sigma}$ set and there exists
		 a countable ordinal $\gamma$ such that $A_{\gamma} = \emptyset$. \end{thm}
		
		\emph{Proof.} First, using the cumulative unions 
		 $G^{0}_n = F^{0}_1 \cup F^{0}_2 \cup ... \cup F^{0}_n$ we express $A_{0}$ by
		 $\bigcup_{n=1}^{\infty} G^{0}_n$, where 
		 $G^{0}_n$ is a nowhere dense and closed subset of $X$. Now, every  $F^{0}_n$ contains a 
		 meager $F_{\sigma}$ subset that is a countable union $\bigcup_{j=1}^{\infty} N^{n}_j$ of
		 closed nowhere dense subsets of $F^{0}_n$. Therefore, we obtain a doubly infinite matrix 
			$(n,j): = N^{n}_j$.\\ 
			 \indent Now (like in the classical proof of the countability of the union of countably many sets) we create the
			 countable sequence $(M^{1}_n)$ of closed nowhere
			dense sets:
			
			 $$ M^{1}_{1} = N^{1}_{1}, M^{1}_{2} = N^{2}_{1} \cup N^{1}_{2},...,
			    M^{1}_{n} = N^{n}_{1} \cup N^{n-1}_{2}\cup...\cup N^{1}_{n},....  $$
					
					Since $N^{n}_{j} \subset G^{0}_n$ (for $j=1,2,...$), we get $ M^{1}_n \subset  G^{0}_n$. 
					 Put now $G^{1}_n = M^{1}_{1} \cup M^{1}_{2} \cup...\cup M^{1}_{n}$, for $n=1,2,...$. In this way we get
					 that $G^{1}_n$ is a nowhere dense,
					 closed subset of $G^{0}_n$ and $A_{1} = \bigcup_{n=1}^{\infty} G^{1}_n$. Moreover, $(G^{1}_n)$
					 is an ascending sequence.\\
					 \indent Now, we can continue inductively the above procedure, using $M^{\alpha}_{n}$ sets instead of $F^{0}_{n}$
					  to define the closed 
					 nowhere dense subsets $N^n_{j}$ of $M^{\alpha}_{n}$. Then we obtain the sets $M^{\alpha+1}_{n}$ and finally the ascending sequence
					 $(G^{(\alpha+1)}_n)$ that satisfies $A_{\alpha + 1} = \bigcup_{n=1}^{\infty} G^{(\alpha+1)}_n$. And, for the limit 
					 ordinals $\lambda$ we put $G^{\lambda}_n = \bigcap_{\beta < \lambda} G^{\beta}_n$. We obtain also, that for any fixed positive integer $n$
                      the transfinite sequence $(G^{\alpha}_n)$ is descending. \\
                     \indent Let $\lambda$ be a limit ordinal, we will show that 
                      $$A_{\lambda} = \bigcap_{\beta < \lambda} A_{\beta} =\bigcup_{n=1}^{\infty}\bigcap_{\beta < \lambda}G^{\beta}_n=\bigcup_{n=1}^{\infty}G^{\lambda}_n.$$
                       Indeed, if $x \in G^{\lambda}_n$,
                     then $x \in G^{\beta}_n \subset A_{\beta}$, for every $\beta < \lambda$ and thus 
                     $x \in \bigcap_{\beta < \lambda} A_{\beta}.$ 
                     On the other hand, if 
                      $x \in A_{\lambda} = \bigcap_{\beta < \lambda} A_{\beta}$, then, by the repeated diagonalization procedure, $x$ belongs
                      to a fixed column $n$ of the shape $(G^{\alpha}_n)$ and hence $x \in G^{\beta}_n$, for every $\beta < \lambda$. Thus the desired equality holds
                       and shows that $A_{\lambda}$ is an $F_{\sigma}$ set.\\
					  \indent Now it is enough to observe that the transfinite sequence $(A_{\alpha})$
						 satisfies the assumptions of Lemma 2.1. Therefore, there exists
					 a countable ordinal $\gamma$ with $A_{\gamma} = \emptyset$.  $\Box$

\section{The $\lambda$-flow} Let $(X,d)$ be compact metric space with Banach-Mycielski content $\lambda$, then for every closed subset $F$ of $X$ and for every positive integer $n$ we define closed ball of radius $\frac{1}{n}$ around $F$ as 
 $\overline{K(F,\frac{1}{n})} = \{x \in X: d(x, F) \leq \frac{1}{n}\}$. Then, using the fact that every closed ball around $F$ is a compact subset of $X$, we define the $\lambda$-\emph{singularity} of $F$ as follows:

 $$ (1) \ \ \ \ \ \  sing(F) = \lim_{n \to \infty} [\lambda(\overline{K(F,\frac{1}{n})})] - \lambda(F)  .$$

Observe that the sequence $(\lambda(\overline{K(F,\frac{1}{n})}))_{n \in \omega}$ is decreasing and bounded by $0$, so the limit exists and is finite. Moreover, by $(b)$ from Section 2, $\lambda(\overline{K(F,\frac{1}{n})}) \geq \lambda(F)$, for every positive integer $n$, thus $sing(F) \geq 0$. In our considerations we will need a refinement of the singularity introduced by $(1)$. Thus let $U$ be an open subset of X and let $F$ be a closed subset of $\overline{U}$, then define the \emph{singularity restricted to} $\overline{U}$, denoted by 
$sing_{\overline{U}}(F)$, as

 $$ (2) \ \ \ \ \ \  sing_{\overline{U}}(F) = \lim_{n \to \infty} [\lambda(\overline{K(F,\frac{1}{n})} \cap \overline{U})] - \lambda(F)  .$$

 \indent Consider now an open subset $U$ of $X$ and define for any positive integer $n$ the closed sets 
 $S_n = \{x \in \overline{U}: d(x, \partial{U}) \geq \frac{1}{n}\}$. We define the $\lambda$-\emph{saturation} of $U$ as follows:

 $$ sat(U) = \lim_{n \to \infty} \lambda(S_n)  .$$

 Observe that the sequence $(\lambda(S_n))_{n \in \omega}$ is increasing and bounded by $\lambda(\overline{U})$, so the limit exists and is not greater than $\lambda(\overline{U})$. We have the following crucial lemma:

 \begin{lem} Let $(X, d)$ be a compact metric space and let $U \subset X$ be an open subset, then

  $$ \lambda(\overline{U}) = \lambda(\partial{U}) + sing_{\overline{U}}(\partial{U}) + sat(U).   $$ \end{lem}
	
	\emph{Proof.} Observe first that $\overline{U} \subset S_m \cup \overline{K(\partial{U},\frac{1}{n})} \cap \overline{U}$, where $n \leq m$. And this implies by $(b)$ and $(c)$ (the properties of $\lambda$ from section 2) that $\lambda (\overline{U}) \leq \lambda(S_m) + \lambda(\overline{K(\partial{U},\frac{1}{n})} \cap \overline{U})$. Hence, by $(2)$ and the fact that $(\lambda(\overline{K(\partial{U}, \frac{1}{n}}))_{n \in \omega}$ is decreasing we get
	
	$$\lambda(\overline{U}) \leq \lambda(S_m) + sing_{\overline{U}}(\partial{U}) +  \lambda(\partial{U}) + \varepsilon,   $$
	
	where $\varepsilon > 0$ is arbitrary and depends on $n$. Now, passing to the limit with $m$ and using the fact that $(\lambda(S_m))_{m \in \omega}$ is increasing, we get:
	
	 $$\lambda(\overline{U}) \leq sat(U) + sing_{\overline{U}}(\partial{U}) +  \lambda(\partial{U}) + \varepsilon.   $$
	
	Since $\varepsilon$ was arbitrary, we get the right-hand inequality.\\
	 \indent For the reverse direction we observe that for $m< n$, $S_m \cup \overline{K(\partial{U}, \frac{1}{n}}) \cap \overline{U} \subset \overline{U}$ and the sets are disjoint. Therefore, by $(b)$ and $(d)$ in section 2, we get
	
	$$ \lambda(S_m) + \lambda(\overline{K(\partial{U},\frac{1}{n}} \cap \overline{U}) \leq  \lambda(\overline{U}). $$
	
	Using the facts that $(\lambda(S_m))_{m \in \omega}$ is increasing and $(\lambda(\overline{K(\partial{U}, \frac{1}{n}}))_{n \in \omega}$
	is decreasing we get that for arbitrary $\varepsilon > 0$ depending on $m$ we have:
	
	  $$ sat(U) - \varepsilon + sing_{\overline{U}}(\partial{U}) + \lambda(\partial{U}) \leq \lambda(\overline{U}).  $$
		
		Since $\varepsilon$ was arbitrary we, get the reverse inequality. $\Box$\\
			
  Let $U$ be an open subset of a compact metric space $X$ then the singularity 
  $sing_{\overline{U}}(\partial{U})$ might not be a monotone function of sets. 
   However, in what follows we will need an invariant that is related to 
    the ascending sequences of closed nowhere dense sets.
 The above lemma suggests a definition of an invariant that coincides with the boundary
  of open set $U$. So for any open subset $U$ of $X$ we define the \emph{cumulation} 
  $cum_{\overline{U}}$ of $U$ as $sing_{\overline{U}}(\partial{U}) + sat(U)$. Some of the properties of cumulations are discussed just before Lemma 3.3.\\

  \begin{lem} Let $U_1,...,U_k$ be pairwise disjoint subsets of a compact metric space $X$.
  Then
   $$cum_{\overline{U_1}}(U_1) + cum_{\overline{U_2}}(U_2)+...+cum_{\overline{U_k}}(U_k) = cum_{\overline{U_1 \cup...\cup U_k}}(U_1 \cup...\cup U_k).$$\end{lem}
  
  \emph{Proof.} We will proceed inductively, so consider first two disjoint open subsets
   $U$ and $V$ of $X$ and observe that $\partial U \cup \partial V = \partial (U \cup V)$ and also $\overline{U} \cap \overline{V} = \partial U \cap \partial V$.\\
     \indent Put now $s_n^A = \lambda(K(\partial(A \cup B), \frac{1}{n})\cap A) - \lambda(\partial A)$ and $s_n^B = \lambda(K(\partial(A \cup B), \frac{1}{n})\cap B) - \lambda(\partial B)$. Therefore we obtain
     
      $$sing_A(\partial A) + sing_B(\partial B) = \lim_{n \to \infty} s_n^A + 
       \lim_{n \to \infty} s_n^A = \lim_{n \to \infty} (s_n^A + s_n^B) = 
        sing_{A \cup B}(\partial A \cup \partial B).$$\\
  \indent Now it is enough to observe that since $U$ and $V$ are disjoint then the sets
    $\{x \in \overline{U}: d(x, \partial A) \geq \frac{1}{n}\}$ and 
    $\{x \in \overline{V}: d(x, \partial B) \geq \frac{1}{n}\}$ are disjoint. Therefore,
    $sat(U \cup V) = sat(U) + sat (V)$ and thus 
    $cum_{\overline{U_1 \cup U_2}}(U_1 \cup U_2)=cum_{\overline{U_1}}(U_1) + cum_{\overline{U_2}}(U_2)$.\\
  \indent To deal with the general case, consider the sets $U_1,...,U_k$ and put 
  $U = U_1 \cup ... \cup U_{k-1}$ and $V = U_k$. Now, it is enough to proceed inductively over the number of the sets $U_i$. $\Box$\\

\vspace{1cm}

Let $U \subset X$ be an open set and let $F$ be a meager $F_{\sigma}$ subset of $\overline{U}$ with the set $W = \{F_n: n=1,2,...\}$ of closed nowhere dense 
 sets witnessing
 that $F$ is $F_{\sigma}$. We can express $F$ using the finite cumulative unions $K_n$
  of the elements of $W$. Thus we obtain an ascending sequence $(K_n)$ of closed nowhere dense subsets of $\overline{U}$ with $F = \bigcup_{n=1}^{\infty} K_n$. 
 Our next step is to extend Lemma 3.1, 
 substituing the nowhere dense set $\partial{U}$ by the meager $F_{\sigma}$ sets. This will allow for the next generalizations concluding at the core of our work that we call the $\lambda$-flow. So observe that Lemma 3.1 implies that:
					
	$$ (3) \ \ \ \ \ \    \lambda(\overline{U}) = \lambda(\overline{V_n}) = \lambda(\partial{V_n}) + sing_{\overline{V_n}}(\partial{V_n})
	 + sat(V_n),   $$
					
	 where $V_n = U \setminus K_n$ and $\partial{V_n} = K_n$. Therefore we obtain that:
	
	  $$  \ \ \ \ \ \   \lim_{n \to \infty} \lambda(\overline{V_n}) = \lambda(\overline{U}). $$
		
	 Cleraly, the definition of the limit is independent of the meager sets $F$ and and the other of creating the cumulative unions. However, in our considerations we will focus
     on the sequences $(a_n)$ and $(b_n)$ with  
        $$a_n = \lambda(K_n), b_n = cum_{\overline{V_n}}(V_n).$$
        
        The property $(a)$ of $\lambda$ implies that $(a_n)_{n \in \omega}$ is increasing and bounded by $\lambda(\overline{U})$ and since 
         $cum_{\overline{V_n}}(V_n) = \lambda(\overline{U}) - \lambda(K_n)$,
          $(b_n)$ is desreasing and bounded by $0$.
        Moreover, since all the summands in $(3)$ are congruence-invariant we get that for any congruence $\sigma$ defined on $\overline{U}$ and $c_n = \lambda(\sigma(K_n))$
         $d_n = cum_{\overline{\sigma(V_n)}}(\sigma(V_n))$ we have

         $$ (4) \ \ \ \ \ \  \lambda(\sigma(\overline{U})) = \lambda(\overline{U}) = 
          \lim_{n \to \infty}a_n + \lim_{n \to \infty}b_n =  
          \lim_{n \to \infty}c_n + \lim_{n \to \infty}d_n  $$\\
    	 
     In what followss we will denote the limit $\lim_{n \to \infty}b_n$ appearing in $(4)$ by $diff(\overline{U}, F)$ and call it
     the $\lambda$-difference between $\overline{U}$ and $F$. Clearly, by the above consideration, we get:\\

      $$(5) \ \ \ \ \ \ \lambda(\overline{U}) =  \lim_{n \to \infty}a_n + diff(\overline{U}, F).$$\\

    We will need the following lemma:
		
		\begin{lem} Let $F = \bigcup_{n=1}^{\infty} F_n$ be an $F_{\sigma}$ subset of a compact metric space $X$ and let $m$ be a
			set function defined for every $F_n$ such that:
			
			 $(i) \ \ \ \ \ m(F_n) \geq 0$;
			
			 $(ii)$ \ \ \ \ \ for any two finite subsets $A$ and $B$ of positive integers such that $A \subset B$, 
			 $m(\bigcup_{i \in A} F_i ) \leq m(\bigcup_{i \in B} F_i)$;
			
			 $(iii)$ \ \ \ \ \ if $F = \bigcup_{n=1}^{\infty} G_n$, where for every positive integer $n$,
			$G_n$ is a finite sum of $n$ elements from $\{F_n: n=1,2,...\}$ and  
			$G_n \subset G_{n+1}$, then the limit
			 $\lim_{n \to \infty} m(G_n) = g$ is a finite positive number.
			   
				 Then the limit is independent of the choice of any monotonic sequence of cumulative sums obtained from the elements of 
				$\{F_n: n=1,2,...\}. $\end{lem} 
				
				 \emph{Proof.} Let $(G_n)$ and $(H_n)$ be two sequences satisfying $(iii)$ with\\
				  $\lim_{n \to \infty} m(G_n) = a$ 
				 and $\lim_{n \to \infty} m(H_n) = b$. Then, we may assume without loss of generality 
				 that $G_1 = F_1$, $G_2 = F_1 \cup F_2$,...,$G_n = F_1 \cup F_2 \cup ... \cup F_n$.
		     Thus there exists a permutation $\pi: \mathbb{N}_{+} \rightarrow  \mathbb{N}_{+}$ such that
				  $H_1 = F_{\pi(1)}$, $H_2 = F_{\pi(1)} \cup F_{\pi(2)}$,...,$H_n = F_{\pi(1)} \cup F_{\pi(2)} \cup ... \cup F_{\pi(n)}$.
					Take now by $(iii)$ an $\varepsilon >0$ such that for $n > N$ one has $m(H_n) > b - \varepsilon$. Hence, by $(ii)$, we get
						
						$$ m(G_{\pi(i)_{max}}) \geq m(H_n) > b - \varepsilon,$$
						
						where $\pi(i)_{max} = max\{\pi(i),...,\pi(n)\}$. And since $\varepsilon$ is arbitrary and the sequence $(G_n)$ is increasing
						then $\lim_{n \to \infty} m(G_n) = a \geq b$.
						 \indent On the other hand, use now the reverse permutation $\pi^{-1}$ with respect to $(G_n)$ and $(H_n)$. Then the symmetry of 
						 the argument gives $b \geq a$. $\Box$\\
    
    Let $F$ be an $F_{\sigma}$ subset of $X$ and let $W = \{F_n: n=1,2,...\}$ be a set of closed sets witnessing
 that $F$ is $F_{\sigma}$. Express $F$ using the finite cumulative unions $K_n$
  of the elements of $W$, obtaining an ascending sequence $(K_n)$ of closed subsets of $X$ with $F = \bigcup_{n=1}^{\infty} K_n$. Observe now that the sequence $(\lambda_1(F_n))_{n \in \omega}$ is increasing and bounded by $\lambda(X) = 1$, so it 
				converges to a limit not greater than $1$. Therefore we can define 
				
				$$  (6) \ \ \ \ \ \  \lambda_1(F) := \lim_{n \to \infty} \lambda(F_n). $$

                The next lemma shows that definition of $\lambda_1$ is 
                independent of the forming of cumulative unions:\\

         \begin{lem} The set function $\lambda_1$ defined on the $F_{\sigma}$ sets is well-defined and congruence-invariant.
		 \end{lem}
			
	\emph{Proof.} First, we observe that, by Lemma 3.3, we get that
		the definition of the limits defined by $(6)$, for the sequences corresponding to different cumulative unions $H_n$ obtained by different permutations $\pi$  of the elements in $\{F_n: n=1,2,...\}$ are all equal. Next, we observe that the congruence-invariance follows from the invariance of $\lambda$ and the fact that
         images of $F_{\sigma}$ sets under congruences are $F_{\sigma}$. $\Box$
		\vspace{1cm}       
    
     Consider now an $F_{\sigma}$ subset $F = \bigcup_{n=1}^{\infty} F_n$ of $X$ with $F_n \subset F_{n+1}$ for any positive integer $n$
		 and assume that $F \subset \overline{U}$, where $U$ is an open subset of $X$.
		We observe that for any such $n$ one can treat $F_n$ as a compact metric space with subspace topology induced by the one from $X$.
		This allows us to define for any closed subset $L$ of $F_n$ the restricted singularity $sing^{F_n}_{\overline{U}}(L)$ and the 
		saturation $sat^{F_n}(V)$ of open subsets $V$ of $F_n$, but still using the Banach-Mycielski 
		 content $\lambda$ defined in $X$:
		
		$$  sing^{F_n}_{\overline{U}}(L):= \lim_{n \to \infty}
		[\lambda(\overline{K(L,\frac{1}{n})} \cap \overline{U} \cap F_n)] - \lambda(L); $$
		
		 $$sat^{F_n}(V) = \lim_{n \to \infty} (\lambda(S_n)),$$
		
		  where $S_n = \{x \in \overline{V}: d(x, \partial{V}) \geq \frac{1}{n}\}$. Now, let each of the $F_n$ contain 
			an $F_{\sigma}$ meager subset 
		 $M_n = \bigcup_{k=1}^{\infty} F^k_{n}$, where $F^k_{n}$ is nowhere dense and closed
         in $F_n $ and $F^k_{n} \subset F^{k+1}_{n}$  for
		 $k=1,2,...$. Moreover, let $M_n \subset M_{n+1}$ for any positive integer $n$. Therefore we can apply Lemma 3.1 with respect to the following
         open sets $U^m_n$ in the subspace topology of $F_n$, 
		  $U^m_n = Int F_n \setminus F^1_{n} \cup F^2_{n} \cup ... \cup F^m_{n}$ and the sequences 
			$(a^m_n)_{m \in \omega}$ and $(b^m_n)_{m \in \omega}$ and with:
			
			$$ (7) \ \ \ \ \ \  a^m_n = \lambda(F^m_n), b^m_n = cum^{F_n} := sat^{F_n}(U^m_n) + sing^{F_n}_{\overline{U^m_n}}(\partial{U^m_n})   $$
			
			 to get that for any congruence $\sigma$ defined on $F_n$ 
             
             $$(8) \ \ \ \ \ \ \ \lambda(\sigma(F_n))=\lambda(F_n) = \lim_{m \to \infty} a^m_n + \lim_{m \to \infty} b^m_n =  \lim_{m \to \infty} c^m_n + \lim_{m \to \infty} d^m_n,$$
			
			  where 
              
              $$ (9) \ \ \ \ \ \  c^m_n = \lambda(\sigma(F^m_n)), d^m_n = cum^{\sigma(F_n)} = sat^{\sigma(F_n)}(\sigma(U^m_n)) + sing^{\sigma(F_n)}_{\overline{\sigma(U^m_n)}}(\partial{\sigma(U^m_n}))   $$\\

         In addition to Lemma 3.4 we observe that for any $F_{\sigma}$ subset $F = \bigcup_{n=1}^{\infty} F_n$ of $X$ 
		with $F_n \subset F_{n+1}$ we get
	  by Lemma 3.1, $(4)$ and $(6)$ and the fact that the sequences in $(7)$
		 are monotonic and bounded (it follows in the same manner like explained just before Lemma 3.3, taking into an account
		 the subspace topology and the restricted singularity $sing^{F_n}$ and the saturation $sat^{F_n}$), that
			$\lambda(F_n) = \lim_{m \to \infty} (a^m_n + b^m_n) = a^n + b^n$. And then, using the
			 fact that the for any positive integer $n$, $M_n \subset M_{n+1}$, we get that
      sets $W_n$ and $W_{n+1}$ of closed nowhere dense subsets witnessing that 
       $M_n$ and $M_{n+1}$ are $F_{\sigma}$ satisfy the following\\

      $$(10) \ \ \ \ \ \ \ \textrm{if} \ A \in W_n \ \textrm{then there exists} \
        B \in W_{n+1} \ \textrm{with} \ A \subset B.$$

        Therefore, for any positive integer $k$ there
        exists a positive integer $l$ such that $a^k_n  \leq a^l_{n+1}$. Now,
         using the fact that $\lim_{m \to \infty} a^m_n = a^n$ and 
         $\lim_{m \to \infty} a^m_{n+1} = a^{n+1}$ we can proceed like in the proof of 
         Lemma 3.3 to get that $a_{n} \leq a_{n+1}$ and the limits are independent 
         of the forming of cummulative unions from $W_1$ and $W_2$, respectively.
         Moreover, by Lemma 3.1, the fact that $\lambda(F_n) \leq \lambda(F_{n+1})$ and
          $a^k_n  \leq a^l_{n+1} $ we get that
             $a^k_n + b^k_n + x = a^l_n + b^l_{n+1} + y$ with $x\geq y \geq 0$, which 
              implies that $b^k_n \leq b^l_{n+1}$. Now, the sequences 
               $(b^{m}_{n})_{m \in \omega}$ and $(b^{m}_{n+1})_{m \in \omega}$ are decreasing and converging to the limits $b^{n}$
               and $b^{n+1}$, respectively. We claim that 
               $b^{n+1} \geq b^n$. Indeed, suppose otherwise and 
              take an $\varepsilon < b^n - b^{n+1}$. Then, by $(10)$, for any $b^k_n$ there exists $b^l_{n+1}$ with $b^k_n \leq b^l_{n+1}$ and $k \leq l$. Take now a positive integer $N$ such that for $l \geq N$,
            $|b^{n+1} - b^l_{n+1}| < \varepsilon$. Then we obtain
            $b^n \leq b^{N}_n \leq b^{N}_{n+1}$ and so 
             $b^N_{n+1}$ is not $\varepsilon$ close to the limit $b^{n+1}$.  
            Thus we get, $b_{n} \leq b_{n+1}$. 
              Therefore the sequences $(a^n)$ and $(b^n)$ are increasing and  
			 bounded by $\lambda_1(X) = \lambda(X)$ and hence we obtain by $(6)$ and $(8)$ the following
			 
             $$ (11) \ \ \ \ \ \  \lambda_1 (F) = \lim_{n \to \infty} (a^n + b^n) = \lim_{n \to \infty} a^n + \lim_{n \to \infty} b^n = a_F + b_F .$$\\

             \indent We observe also that $(10)$ allows us to infer
             that the functions $f_1(F_n) = a^n$ and $f_2(F_n) = b^n$ satisfy the assumptions of Lemma 3.3  and thus the limits $a_F$ and $b_F$ are independent of the choice of ascending sequences $(F_n)$ 
				 of closed nowhere dense sets $F_n$ forming $F$.\\
             
           \indent In this way, defining the set $F^{(1)} = \bigcup_{n=1}^{\infty}\bigcup_{m=1}^{\infty} F^m_n$ we obtain an $F_{\sigma}$
				set that is meager in the subspace topology of $F$. The above arguments
                show that $a_F = \lim_{n \to \infty}\lim_{m \to \infty} \lambda(F^m_n) = \lambda_1(F^{(1)})$. Since $a_F = \lambda_1(F^{(1)})$ and
                 $\lambda_1(F) = a_F + b_F$ we will adopt the following notation:
                 $diff_1(F, F_1) = \lambda_1(F) - \lambda_1(F^{(1)}) = b_F$, and we call
                 $diff_1(F, F_1)$, the $\lambda_1$-\emph{difference between} $F$ \emph{and} $F_1$.
                 It is clear that the $\lambda_1$-difference describes the difference between
                 the $\lambda_1$ values of an $F_{\sigma}$-set $F$ and any other $F_{\sigma}$-set contained in $F$ that is an ascending union of $F_{\sigma}$-sets. Moreover, since every congruence preserves $F_{\sigma}$-sets, then the
                 $\lambda_1$-difference is congruence-invariant.\\
                 \indent Therefore we can write:
				
				 $$(12) \ \ \ \ \ \ \lambda_1(F) = \lambda_1(F^{(1)}) + diff_1(F, F_1).$$\\

				 Let $X$ be a compact metric space and let $(F^{(\alpha)})_{\alpha \leq \gamma}$ be a transfinite descending
				 sequence of $F_{\sigma}$ subsets of $X$ that satisfies the assumptions
                 of Theorem 2.3 and stabilizes at a countable ordinal $\gamma$ 
				 (that is $F^{(\gamma)} = \emptyset$). In the next step we will link together
                 Theorem 2.3 with $\lambda_1$ and the $\lambda_1$-difference between consecutive terms of the sequence $(F^{(\alpha)})_{\alpha \leq \gamma}$.
                 In this way we will obtain one more invariant that is the core of our geometrization. The invariant can be described as the transfinite series  that uses $\lambda_1$-differences of consecutive terms of the transfinite sequence and
                 the quantity we will define momentarily that describes a possible discontinuity of $\lambda_1$ at the limit ordinals. \\
                 \indent Observe now that since the sequence $(F^{(\alpha)})_{\alpha \leq \gamma}$ is descending, then the terms $F^{(\alpha)}$ and $F^{(\alpha)}$ ($\alpha < \beta$) of the sequence satisfy a modification of $(10)$:
                  let $W_{\alpha}$ and $W_{\beta}$ be the sets of closed subsets of $X$ witnessing that $F^{(\alpha)}$ and $F^{(\alpha)}$ are $F_{\sigma}$ sets,
                  then if $A \in W_{\beta}$ then there exists a $B \in W_{\alpha}$ with 
                  $A \subset B$. Thus we can use the same argument as in the proof of Lemma 3.3 to infer that $(\lambda_1(F^{(\alpha)}))_{\alpha \leq \gamma}$ is decreasing.\\
                  \indent Let $\tau$ be a limit ordinal less than
                  $\gamma$, since $(\lambda_1(F^{(\alpha)}))_{\alpha \leq \gamma}$ is a transfinite sequence of non-negative real numbers, then $g^{\tau} :=\inf_{\alpha < \tau}(\lambda_1(F^{(\alpha))} \geq 0$. We will say that there is a \emph{discontinuity} at $\tau$, denoted
                 $disc(F^{(\tau)})$ if $g^{\tau} - \lambda_1(F_{\tau}) > 0$. Observe
                  that since $(\lambda_1(F^{(\alpha)}))_{\alpha \leq \gamma}$ is decreasing, then this kind of discountinuity is possible.\\
                    \indent We are ready to define the \emph{transfinite series corresponding to}
                   $(F^{(\alpha)})_{\alpha \leq \gamma}$ as follows:
                 \[S_{\beta} = 
                    \begin{cases}
                        0, & \beta=0,\\[2pt]
                      S_\alpha+ diff_1(F^{(\alpha)}, F^{(\alpha+1)}), & \beta=\alpha+1,\\[2pt]
                    \displaystyle\sup_{\alpha<\beta}S_\alpha +disc(F^{(\beta)})
                    & \beta \ \text{is a limit ordinal}.
                    \end{cases}\]\\

                   Denote now the number $S_{\gamma}$ by $\lambda_2(F^{(0)})$. We have the following lemma:\\
                    
                   \begin{lem} Let $X$ be a compact metric space and let $(F^{(\alpha)})_{\alpha \leq \gamma}$ be a decreasing sequence 
				   of $F_{\sigma}$ sets that satisfies the assumptions of Theorem 2.3 and stabilizes at a countable ordinal $\gamma$, then the following holds:
					
					$$\lambda_2(F^{(0)}) = \lambda_1(F^{(0)}).$$  \end{lem}

                 \emph{Proof.} First, we apply $(12)$ to get that for every natural $n \geq 1$ we have, $\lambda_1(F^{(0)}) = \sum_{i=1}^n diff_1(F^{(0)}, F^{(n)}) + \lambda_1(F^{(n)}) = S_n + \lambda_1(F^{(n)})$ Hence, using the fact that $(\lambda_1(F^{(\alpha)}))_{\alpha \leq \gamma}$ is decreasing, we obtain 
                 $\sup_{n < \omega}(\lambda_1(F^{(0)}) - \lambda_1(F^{(n)}) = \lambda_1(F^{(0)} - \inf_{n < \omega} (\lambda_1(F^{(n)})) = \lambda_1(F^{(0)}) -g^{\omega}.$ Therefore 
                 $S_{\omega} = \sup_{n < \omega} S_n + g^{\omega} - \lambda_1(F(\omega)) =
                  \lambda_1(F^{(0)}) - \lambda_1(F^{(\omega)}).$\\
                    \indent Next we can proceed with the above procedure inductively, changing 
                     $\omega$ to the limit ordinals $\beta \leq \gamma$ and the natural numbers $n$ to the successor ordinals $\alpha < \beta$. Hence we obtain $S_{\beta} = \sup_{\alpha < \beta} S_{\alpha} + g^{\beta} - \lambda_1(F(\beta)) =
                  \lambda_1(F^{(0)}) - \lambda_1(F^{(\beta)}).$ Now, it is enough to
                   observe that by Theorem 2.3, we have $\lambda_1(F(\gamma)) = 0$ and thus $S_{\gamma} = \lambda_1(F^{(0)})$. $\Box$\\
                    
						\vspace{1cm}
                        
						The above series of lemmas allows us to define the $\lambda$-\emph{flow}:
						
						\begin{defin} Let $X$ be a compact metric space and let
						  $(F^{(\alpha)})_{\alpha \leq \gamma}$ be a decreasing sequence of $F_{\sigma}$ sets satisfying the assumptions of Theorem 2.3 and stabilizing at a countable ordinal $\gamma$,
						 then the number
						
						$$(13) \ \ \ \ \ \  \lambda_T(X) :=  diff(X, F^{(0)}) + \lambda_2(F^{(0)}) $$
						
						is called the $\lambda$-flow determined by $F^{(\alpha)}_{\alpha \leq \gamma}$.  \end{defin}
						
							\vspace{1cm}
                        
                        The following theorem is fundamental in the proof of the Banach-Ulam conjecture and in our solution of the Talagrand's problem:
						
						\begin{thm} Let $X$ be a compact metric space and let
						  $(F^{(\alpha)})_{\alpha \leq \gamma}$ be a decreasing sequence of $F_{\sigma}$ sets satisfying the assumptions of Theorem 2.3 and stabilizing at a countable ordinal $\gamma$,
                           then $\lambda_T(X) = 1$. \end{thm}
						
						\emph{Proof.} First, we use $(5)$ to get that 
                        $\lambda(X)  = diff(X, F^{(0)}) + \lim_{n \to \infty}a_n$. But  $(a_n)$ is just an increasing sequence defining $\lambda_1(F^{(0)})$. Therefore we get $\lambda(X) =  diff(X, F^{(0)}) + \lambda_1(F^{(0)})$. Now, it is enough to apply Lemma 3.6.  $\Box$
						
						 \begin{cor} The $\lambda$-flow is congruence invariant. \end{cor}
						
						\emph{Proof.} Let $F^{(\alpha)}_{\alpha \leq \gamma}$ be a decreasing sequence stabilizing at $\gamma$ 
						It is enough to observe that congruences $\sigma$ are homeomorphisms from $F^{(\alpha)}$ onto $\sigma(F^{(\alpha)})$ and 
						 apply Theorem 3.7. $\Box$\\

                         In addition to Corollary 3.8 we will show in the Claim of Theorem 4.1 that $\lambda$-flow is invariant under equidecomposability defined on page 4. And this naturally extends to Tarski's operations. Therefore we obtain a natural tool for blocking any paradoxical decompositions of compact metric spaces using pieces belonging to the regular algebra. 
						
                    \vspace{1cm}
                        
						 \indent Recall that by a geometrization of the compact metric space $X$ we mean the set 
						 $\mathcal{G} =\{ \mathcal{C}, \mathcal{I}\}$, where $\mathcal{C}$ is the set of all congruences defined on the closed 
						 subsets of $X$ and $\mathcal{I}$ is the set of invariants of the congruences in $\mathcal{C}$. It should be stressed that the invariance pertains to images of the closed sets and is not taken globally. We have the straightforward
							corollary to our Theorem 3.8:
							
							\begin{cor} Let $X$ be a compact metric space, then the set 
							 $$\{ \mathcal{C}, sing, sat, cum, diff, \lambda, \lambda_1, \lambda_2, \lambda_T(X) \}$$
								is a geometrization of $X$.  \end{cor}

\section{The Banach-Ulam conjecture and\\
 Talagrand's problem}

 \begin{thm} Let $X$ be a compact metric space, then there is no positive integer $n$ such that

       $$ (14) \ \ \ \ \ \  (n+1)[X] = n[X] \ \ \ \textrm{in}\ \  Reg(X).$$ \end{thm}
			
			\emph{Proof.} Suppose to the contrary that there exists a positive integer $n$ satisfying $(14)$. Therefore, there exist
			$n$ disjoint copies of $X$ and then aditional $(n+1)$ copies of $X$ with the Tarski's operations between the $n$ copies and
			the $n+1$ copies. So, let $\{P_{ij}\}_{i=1}^{k_j}$ be a regular partition of the $j$-th copy, where $j=1,2,...,n$. Then, there
			are congruences $\sigma_{ij}$ such that $\{\sigma_{ij}(P_{ij})\}_{i=1}^{k_j}$, $j = 1,...,n$ forms a regular partition of $n+1$
			copies of $X$. In what follows we extend the $\lambda$-flow, defined initially on $X$, to the disjoint copies of $X$ and treat
			each of the flows independently for each of the copies. Thus, for every positive integer $n$, Theorem 3.7 allows us to use the
			independence and define for the extended $\lambda$-flow the number $\lambda_{T}(Y_n) = n \cdot \lambda(X)$, where $Y_n$ is the
			 union of $n$ disjoint copies of $X$.
			We will show that the $n$ regular partitions of $n$ copies of $X$ and the congruence invariance of the extended $\lambda$-flow,
			induces by $(14)$ a regular partition on $n+1$ copies of $X$ with extended $\lambda$-flow that equals $n \cdot \lambda_T(X) = n$.
			On the other hand, Theorem 3.7 implies that that every extended $\lambda$-flow defined by any regular partition of $n+1$ copies of $X$
			 must equal $(n+1) \cdot \lambda_T(X) = n+1$, thus giving the desired contradiction.\\
       \indent Recall that every regular partition $\{P_i\}_{i=1}^k$ defines pairwise disjoint regular open subsets $U_i$ and the boundaries
			 $\partial{U_i}$ of the $U_i$ such that $\bigcup_{i=1}^k \partial{U_i}$ is a nowhere dense closed subset of $X$ and that
			 $X \setminus \bigcup_{i=1}^k \partial{U_i} = \bigcup_{i=1}^k U_i$. Since $P_i = U_i \triangle M_i$, where 
			  $M_i= \bigcup_{n=1}^{\infty} N^{i}_n$ and $N^{i}_n$ is a nowhere dense subset of $X$, we can define $\overline{N^{i}_n}$ that
				is also nowhere dense and then a meager $F_{\sigma}$ set $\widehat{M_i} = \bigcup_{n=1}^{\infty} \overline{N^{i}_n}$.
				 Hence we obtain the following $F_{\sigma}$ set: $F = \bigcup_{i=1}^k \partial{U_i} \cup \bigcup_{i=1}^k \widehat{M_i}$.
				 Clearly, we have $\bigcup_{i=1}^k \partial{U_i} \cup \bigcup_{i=1}^k {M_i} \subset F$ and the sets in
				 $\{U_i\setminus F\}_{i=1}^k$ are pairwise disjoint.\\
				 \indent Now, using the cumulative sums of elements in  $\{\partial{U_i}\}_{i=1}^k$ and the ones defining $\widehat{M_i}$
				for $i=1,...,k$ we can express $F$ as $\bigcup_{n=1}^{\infty} F_n$ with $F_n \subset F_{n+1}$ and $F_n$ being closed 
				 and nowhere dense.\\
				 \indent Since $\{P_i\}_{i=1}^k$ is regular, then $P_i \cap F_n$ has the property of Baire
				 for any positive integer $n$ and thus $P_i \cap F_n = U^{F_n}_i \triangle M^{F_n}_i$, where
				$M^{F_n}_i = \bigcup_{m=1}^{\infty} F^{F_n}_m$ is meager set in the subspace
					topology of $F_n$. Also, the set $\widehat{M^{F_n}_i} = \bigcup_{m=1}^{\infty} \overline{F^{F_n}_m}$ is a meager $F_{\sigma}$ 
					subset of $F_n$.
					We obtain also in this way the sets $U^{F_n}_i$ that are pairwise disjoint regular open in the subspace topology of $F_n$,
				 and the boundaries $\partial{U^{F_n}_i}$ of these open subsets, also contained in $F_n$. Thus we get
				
				 $$ (15) \ \ \ \ \ \ \ \ \ \  \ \  F_n = \bigcup_{i=1}^k U^{F_n}_i \cup  \bigcup_{i=1}^k  \partial{U^{F_n}_i},$$
				
				 where $\bigcup_{i=1}^k U^{F_n}_i$ and $\bigcup_{i=1}^k  \partial{U^{F_n}_i}$ are disjoint. In this way we obtain a meager
				 $F_{\sigma}$ subset $A$ of $F$:
				
				         $$ (16) \ \ \ \ \ \  A = \bigcup_{n=1}^{\infty} \bigcup_{i=1}^k (\partial{U^{F_n}_i} \cup\widehat{M^{F_n}_i}).$$
		      
             We obtain also the pairwise disjoint sets $U^F_i = \bigcup_{n=1}^{\infty} U^{F_n}_i$ 
						open in the subspace topology of $F$, disjoint with $\bigcup_{n=1}^{\infty} \bigcup_{i=1}^k \partial{U^{F_n}_i}$
						and such that

						$$(17) \ \ \ \ \ \  \bigcup_{n=1}^{\infty} \bigcup_{i=1}^k \partial{U^{F_n}_i} \cup \bigcup_{i=1}^k U^F_i = F.$$

							\indent Put now, $F^{(0)} = F$ and $F^{(1)} = A$. It is clear that we can continue in the way as above, using the regularity
							 of $\{P_i\}_{i=1}^k$ to get a decreasing transfinite sequence $(F^{(\alpha)})$ such that $F^{\alpha+1}$ is 
							 a meager $F_{\sigma}$ subset 
							 of $F^{\alpha}$ and the assumptions of Theorem 2.3 are satisfied. Therefore, for the limit ordinals $\beta$, we get that $ F^{\beta} = \bigcap_{\alpha < \beta} F^{\alpha}$
							  is also a meager $F_{\sigma}$ subset of any $F^{(\alpha)}$. Thus, again by Theorem 2.3, there exists a countable ordinal $\gamma$ such that $F^{\gamma} = \emptyset$. And hence we can apply Theorem 3.7
							 and obtain that the $\lambda$-flow determined by each of the regular partition $\{P_{ij}\}_{i=1}^{k_j}$ of $X$ equals 
							  $ \lambda(X) = 1$. But the $n$ partitions $\{P_{ij}\}_{i=1}^{k_j}$ form jointly a partition of $n$ disjoint
								 copies of $X$. And therefore we get that the extended $\lambda$-flow determined by the partitions equals 
								 $n \cdot \lambda(X) =n$.\\
								 \indent Our next step is to observe that the extended $\lambda$-flow defined on $n$ copies of $X$ and the 
								 fact that $(14)$ holds,
								 imply that the following claim is true:\\
								
								\textbf{Claim.} $(14)$ implies the existence of an extended $\lambda$-flow on $(n+1)$ disjoint copies of $X$ that equals
								  $n \cdot \lambda(X)$.
									\vspace{1cm}
                                    
	\emph{Proof.} Let $F = \bigcup_{n=1}^{\infty} F_n$ be a meager $F_{\sigma}$ subset of $X$ such that $F_n \subset F_{n+1}$. We will modify $(11)$ in the way it fits to
    Tarski's operations. In aiming this we observe that any regular partition $\{P_i\}_{i=1}^k$ implies by $(15)$ that  $F_n = \bigcup_{i=1}^k U^{F_n}_i \cup  \bigcup_{i=1}^k  \partial{U^{F_n}_i}$, where $U^{F_n}_i$ are pairwise disjoint relatively open subsets of $F_n$, and $\partial{U^{F_n}_i}$ are boundaries of $U^{F_n}_i$, included in $F_n$. Let $\mathcal{C} = \{\sigma_1,...,\sigma_k\}$ be a fixed set of congruences of $X$ such that for every positive integer $n$, $U^{F_n}_i$ is a domain of some element $\sigma_l$ of $\mathcal{C}$. We will write
    $U^{F_n}_{\sigma_{l}}$ to denote the domain. Clearly, 
    $F^{\sigma_l} = \bigcup_{n=1}^{\infty} \overline{U^{F_n}_{\sigma_{l}}}$ is also 
      a meager $F_{\sigma}$ subset of $X$ such that $\overline{U^{F_n}_{\sigma_{l}}}
      \subset \overline{U^{F_{n+1}}_{\sigma_{l}}}$. We assume additionally that there
       exists a meager $F_{\sigma}$ subset $M = \bigcup_{n=1}^{\infty} M^{F_n}_{\sigma_l}$ of $F^{\sigma_l}$, where $M^{F_n}_{\sigma_l} \subset \overline{U^{F_n}_{\sigma_{l}}}$.
      Therefore, by $(11)$, we get 
       
        $$ (18) \ \ \ \ \ \  \lambda_1 (F^{\sigma_l}) = \lim_{n \to \infty} (a_{\sigma_l}^n + b_{\sigma_{l}}^n) = \lim_{n \to \infty} a^n_{\sigma_{l}} + \lim_{n \to \infty} b_{\sigma_l}^n = a_{F_{\sigma_{l}}} + b_{F_{\sigma_{l}}}, $$\\    
        where the numbers $a_{\sigma_l}^n,  b_{\sigma_{l}}^n, a_{F_{\sigma_{l}}}$ and $b_{F_{\sigma_{l}}}$ are defined mutatis mutandis regarding the ones defined in $(11)$. Moreover, since $\lambda(\overline{U^{F_n}_{\sigma_{l}}}) = \lambda(\sigma_l(\overline{U^{F_n}_{\sigma_{l}}}))$ we get 
        $\lim_{n \to \infty} b_{\sigma_l}^n =\lim_{n \to \infty} c_{\sigma_l}^n =c_{F_{\sigma_{l}}}$, where $ c_{\sigma_l}^n = \lambda(\sigma_l(\overline{U^{F_n}_{\sigma_{l}}}))$.\\
         \indent On the other hand, recalling $(4)$ and the construction just before $(11)$ we get, by Lemma 3.2, that for the sequences $(b^{\sigma_{l}}_n)$ defining $c_{\sigma_{l}}^n$ the following holds:\\
         
         $$ (19) \ \ \ \ \  \ \ \  \ \ \sum_{l=1}^k b^{\sigma_{l}}_n = b_n,$$ 
         where $b_n = cum^{F_n}_{\overline{U^{F_n}_1 \cup...\cup U^{F_n}_k}}(U^{F_n}_1 \cup...\cup U^{F_n}_k)$.\\

           Therefore we obtain by $(10)$ and $(19)$ the following 
        
        $$(20) \ \ \ \ \ \  \lambda_1 (F) = a_F + b_F= a_F + b^{\sigma_1}_F +...+b^{\sigma_k}_F
         = a_F + c^{\sigma_1}_F +...+c^{\sigma_k}_F.$$

          \indent Let $(F^{(\alpha)})$ be the transfinite sequence of $F_{\sigma}$ sets defined by repepetitive usage of $(16)$ and $(17)$ that stabilizes at a countable ordinal $\gamma$. So we can express each $F^{(\alpha)}$ as $\bigcup_{i=1}^{\infty}F_n^{(\alpha)}$, where $F_n^{(\alpha)} \subset F_n^{(\alpha + 1)}$. Moreover, every $F_n^{(\alpha)}$ can be partitioned up to the nowhere dense boundaries into pairwise disjoint relatively open subsets $U^{F^{(\alpha)}_n}_{\sigma_{l}}$. Each of the open sets is a domain for some congruence $\sigma_l \in \mathcal{C}$. Assume now that the following holds:\\

    \[
\begin{gathered}
(21)\quad \text{for any } \alpha_1,\alpha_2\leq\gamma,
l_1,l_2\in\{1,\ldots,k\},
\text{if } \alpha_1\neq\alpha_2
\text{ or } l_1\neq l_2,\\
\text{then}\\
\sigma_{l_1}\left(
U^{F^{(\alpha_1)}_n}_{\sigma_{l_1}}
\setminus F^{(\alpha_1+1)}
\right)
\cap
\sigma_{l_2}\left(
U^{F^{(\alpha_2)}_m}_{\sigma_{l_2}}
\setminus F^{(\alpha_2+1)}
\right)
=\emptyset.
\end{gathered}
\]      
         
We observe that the disjointness provided by $(21)$ allows us to define the limits appearing in $(20)$. Thus we can define now the \emph{transfinite series induced by}
                   $(F^{(\alpha)})_{\alpha \leq \gamma}$ and $\sigma_1,...,\sigma_l$ as follows:
                   \[T_{\beta} = 
                    \begin{cases}
                        0, & \beta=0,\\[2pt]
                      T_\alpha+ c^{\sigma_1}_{F_{\alpha}} +...+c^{\sigma_k}_{F_{\alpha}}, & \beta=\alpha+1,\\[2pt]
                    \displaystyle\sup_{\alpha<\beta}T_\alpha +disc(F^{(\beta)})
                    & \beta \ \text{is a limit ordinal.}
                    \end{cases}\]\\
           \indent We will show that $T_{\gamma} = \lambda_1(F^{(0)}).$ Indeed, for the transfinite series $S_{\alpha}$ corresponding to $(F^{(\alpha)})$, we get
            $S_0 = T_0$. Also, $(12)$, $(19)$ and $(20)$ and $(21)$ imply that $T_{\beta} = S_{\beta}$ for $\beta = \alpha +1$. And thus we get 
             $\sup_{\alpha<\beta}T_\alpha = \sup_{\alpha<\beta}S_\alpha$. Finally, we obtain that $(20)$ forces that the $disc(F^{(\beta)})$ is well-defined in the case of $T_{\beta}$. Therefore we have $T_{\alpha} = S_{\alpha}$ for every $\alpha \leq \gamma$ and thus, by Lemma 3.5, $T_{\gamma} = S_{\gamma} = \lambda_2(F^{(0)}) = \lambda_1(F^{(0)})$.\\

           \indent  Consider now the transfinite sequence $(F^{(\alpha)}_j)$ of meager $F_{\sigma}$ sets 
	$F^{(\alpha)}_j$ in the $j$-th copy of $X$ ($j=1,...,n$) defined just after $(17)$.
     By Theorem 2.3 we may assume without loss of generality that each of $(F^{(\alpha)}_j)$ stabilizes at a countable ordinal $\gamma$. 
    Recall that in each step $\alpha$ of the construction of  meager $F_{\sigma}$ sets 
	$F^{(\alpha)}_j$ in the $j$-th copy of $X$ ($j=1,...,n$), we use an ascending sequence $(F_{nj}^{(\alpha)})$ of closed nowhere dense subsets of $F^{(\alpha)}_j$.
    And each set $F_{nj}^{(\alpha)}$ splits modulo the boundaries into finitely many relatively open sets $U^{F^{(\alpha)}_{nj}}_{\sigma_{jl}}$.
     Since every congruence extends from its domain to the closure, each of the sets $\overline{U^{F^{(\alpha)}_{nj}}_{\sigma_{jl}}}$
	is a domain for some congruence $\sigma_{jl}$ (here we extend from the comeager set defining 
	  $U^{F^{(\alpha)}_{nj}}_{\sigma_{jl}}$).
      Also, $\sigma_l$ is a homeomorphism from $\overline{U^{F^{(\alpha)}_{nj}}_{\sigma_{jl}}}$
	onto $\sigma_l(\overline{U^{F^{(\alpha)}_{nj}}_{\sigma_{jl}}})$. Therefore, $\sigma_l(U^{F^{(\alpha)}_{nj}}_{\sigma_{jl}})$ is an open subset of $\sigma_l(\overline{U^{F^{(\alpha)}_{nj}}_{\sigma_{jl}}})$. Note that it what follows the congruences $\sigma_{jl}$ induce the Tarski's operations between
     $n$ and $n+1$ copies of $X$.\\
     \indent Since the sets $U^{F^{(\alpha)}_{nj}}_{\sigma_{jl}}$ are pairwise disjoint then, by $(14)$, the relatively open
		 sets $\sigma_l(U^{F^{(\alpha)}_{nj}}_{\sigma_{jl}})$ are pairwise disjoint or one of the images is included in a closed nowhere dense subset $A$ of the other, beeing simultaneously relatively open in $A$.
           Thus we obtain for $Y_{n+1}$, the union of $(n+1)$ disjoint copies of $X$ the following:\\

           $$(22) \ \ \ \ \ \   Y_{n+1} = \bigcup_{j=1}^n \bigcup_{l=1}^k (C_{jl} \cup \bigcup_{\alpha < \gamma} C^{(\alpha)}_{ij}),$$ where $C_{jl}$ are pairwise disjoint images of
            the comeager subsets $U_{jl} \setminus F^{(0)}_j$ of $X_j$ and $U_{jl}$ is a domain for the congruence $\sigma_{jl}$. Moreover, $C^{(\alpha)}_{ij}$
            are pairwise disjoint images of the comeager subsets  $U^{F^{(\alpha)}_{j}}_{\sigma_{l}} \setminus F_j^{(\alpha + 1)}$ of $F_j^{(\alpha)}$. And since the partition defined by $(14)$ is regular, all
             the summands appearing in $(22)$ are pairwise disjoint and thus $(21)$ is satisfied. This allows us to define $n$ transfinite series induced by
                   $(F_j^{(\alpha)})_{\alpha \leq \gamma}$ and $\sigma_1,...,\sigma_{l_j}$ in each of the $j$-th copies. And for each of the series $(T^j_{\alpha})$ we obtain 
                   $T^{j}_{\gamma} = \lambda_2(F_j^{(0)}) = \lambda_1(F_j^{(0)})$.
                    Therefore, by the fact that images of $F_{\sigma}$ sets under congruences are $F_{\sigma}$ and by $(11)$, $(12)$, $(20)$ and $(22)$ we get that  
                    $\lambda_{T}(Y) =   n \cdot \lambda(X).$\\

    \indent Finally, we can repeat all the above arguments involving the transfinite sequences 
	 of $F_{\sigma}$ meager sets
and the $\lambda$-flows determined by any $(n+1)$ regular partitions of $n+1$ copies of $X$. In this way
we obtain that any such regular partition defines the extended $\lambda$-flow that is 
equal $(n+1) \cdot \lambda(X).$ Therefore we arrive to a contradiction, starting from $n$ disjoint
	 copies of $X$ and assuming $(14)$ we obtain that there exists an extended $\lambda$-flow defined by a 
	 regular partition of $n+1$ disjoint copies of $X$ that is equal to $n \cdot \lambda(X)$. 
	On the other hand, any regular partition of such $n+1$
	copies defines the $\lambda$-flow that equals $(n+1) \cdot \lambda(X)$. $\Box$
		
        \vspace{1cm}												
		The straightforward application of the Tarski's Theorem 2.1 in the case of $\mathcal{S}(Reg(X))$
		confirmes the Banach-Ulam conjecture and simultaneously solves the Talagrand's problem :
														
		\begin{thm} Let $X$ be a compact metric space, then there exists a finitely additive
		probability measure defined on the regular algebra $Reg(X)$ that is congruence invariant.\end{thm}

	\section{Geometrization and the measures} Let $(X,d)$ be a compact metric space, we will say that $X$ is \emph{regular}
								if the singularity of every closed subset of $X$ is equal $0$. In this section we will show that the regularity
								 of compact metric spaces implies the existence of congruence-invariant Banach-Mycielski measures.
								
								 \begin{thm} Let $X$ be a compact metric space, then $X$ is regular if and only if any Banach-Mycielski
								 measure agrees with the Banach-Mycielski content $\lambda$ on closed sets. \end{thm}
								
								\emph{Proof.} For the forward direction assume that $X$ is regular. Then for every closed subset $F$ and every
								$\varepsilon > 0$ there exists a positive integer $N$ with 
								 $\lambda(\overline{K(F, \frac{1}{n})}) < \lambda(F) + \varepsilon$ for $n > N$. Now, every Banach-Mycielski
								 measure $\mu$ satisfies $\lambda(F) \leq \mu(F)$. So suppose to the contrary that $\lambda(E) < \mu(E)$ for
								 some closed subset $E$ of $X$. Then we get $\mu(E) - \lambda(E) = a >0$. But then we can take an $\varepsilon < a$
								  and thus we get 
									
									  $$\mu^{*}(E) = \mu(E) \leq \lambda(E) + \varepsilon < \lambda(E) + a$$
										
										and so $\mu(E) - \lambda(E) < a$, a contradiction. Note that above we use the following:
										
										$\lambda(\overline{K(F, \frac{1}{n})}) < \lambda(F) + \varepsilon$ implies that the 
										 $sup\{\lambda(K): K \subset K(K, \frac{1}{n}), K = \overline{K} \} \leq \lambda(F) + \varepsilon$.
									   
								   \indent For the backward direction assume that $\mu$ is a Banach-Mycielski measure defined in $X$ that 
									 agrees with $\lambda$ on closed subsets of $X$. Then for any closed subset $F$ of $X$ we have
									
									   $$\mu(F) = inf\{\mu_{0}(U): F \subset U, U \ \textrm{open in} \ X\}.$$
										
									Thus, take an open subset $U$ of $X$ with $F \subset U$ and $\mu_{0}(U) \leq \mu(F) + \varepsilon$ 
									 ($\mu_0$ is defined in Sect. 2) for some
										$\varepsilon > 0$. Then, consider the closed set $K = X \setminus U$. Clearly, $F$ and $K$ are disjoint
										 and so their distance $d(F,K) = a$ is positive. Pick now a closed ball $\overline{K(F, \frac{1}{n})}$,
										 where $n$ is a positive integer so large that $\frac{1}{n} \leq \frac{a}{2}$. We observe that, by
										 the definition of $a$, the set $\{x \in X: d(x, F) < a\}$ is contained in $U$ and so 
										 $\overline{K(F, \frac{1}{n})} \subset U$.\\
										 \indent Now, $\mu_{0}(U) = sup \{\lambda(L): L \subset U, L = \overline{L}\} \leq \mu(F) + \varepsilon$.\\
							        \indent Thus we get
											
											   $$ \lambda(\overline{K(F, \frac{1}{n})}) \leq \mu(F) + \varepsilon = \lambda(F) + \varepsilon.$$
												
		 Letting $\varepsilon \rightarrow 0$ we get $\lim_{n \to \infty} \lambda(\overline{K(F, \frac{1}{n})}) = \lambda(F)$.
				$\Box$
							\vspace{1cm}
                                                    
				The next theorem relates our geometrization of compact metric spaces with
								            the main result of Bandt-Baraki [3] and explains the construction of Davies and Ostaszewski [4]:

                             \begin{thm} Let $X$ be a regular compact metric space, then every probabilistic Banach-Mycielski measure $\mu$ is
					congruence-invariant countably additive measure. \end{thm}
														
				\emph{Proof.} Observe that since $\mu$ is finite, then it is inner regular. Thus for every Borel set we apply
				Theorem 6.1 and get: $\mu(B) = sup\{\lambda(F): F \subset B, F = \overline{F}\}$.\\
				\indent Suppose now that $C$ is congruent to $B$, that is, there exists a congruence $\sigma$ such that
			$\sigma(B) = C$. Then also $\sigma(F)$ is congruent to $F$ and since $F$ is closed we get by the congruence-
			invariance of $\lambda$ that $\lambda(\sigma(F)) = \lambda(F)$. Therefore

  \[
   \begin{split}
    \mu(\sigma(B)) &= sup\{\lambda(K): K \subset \sigma(B), K = \overline{K}\} =\\
         &=sup\{\lambda(F): \sigma(F) \subset \sigma(B), F = \overline{F}\} = \mu(B).
\end{split}
\]   $\Box$
																	
	It is interesting to observe that Banadt and Baraki proved that in the hyperspaces of compact
	subsets of given compact metric space without isolated points, equipped 
	with Hausdorff metric, there is no probabilistic 
	Banach-Mycielski that is congruence-invariant and countably additive ([3], Thm. 2). Also, there cannot
	any such measure in any countable compact metric spaces. Therefore, by Theorem 5.2, we observe that the
	aforementioned spaces must contain some closed subsets with positive singularity. In fact, it is easy to
	verify using the Cantor-Bendixson derivative, that the finite non-empty set of points in 
	the Cantor-Bendixson process, one step before it stabilizes, is just the set of points with positive
	singularity.

           \section{Conjecture of Ulam} We will show in this section an application of Theorem 6.2 to the conjecture of Ulam on the
					 product measure in the Hilbert cube $I^{\omega}$. Recall that the problem asks if the standard product Lebesgue $\mu$ is
					 congruence-invariant, where $I^{\omega}$ is treated as a compact metric space with the following metric $d_a$:
					
					   $$d_a(x,y) = (\sum_{i = 1}^{\infty}a^{2}_i(x_i - y_i))^{1/2},$$
						
						  where $(a_i)$ is a sequence of positive real numbers with $\sum_{i=1}^{\infty}a^{2}_i < \infty$. Some partial results were
							obtained by J. Mycielski [11], J. Ficket [6], and S.M Jung gave finally a quite complex solution to this conjecture in 
							 [7]. We will show that the conjecture follows from Theorem 5.2:
							
							  \begin{thm} The standard product Lebesgue $\mu$ is congruence-invariant. \end{thm}
								
								\emph{Proof.} First, we show that the compact metric space $(I^{\omega}, d_a)$ is regular. So consider a closed subset
								 $F$ of $I^{\omega}$ and the closed ball at $F$, $\overline{K(F, \frac{1}{n})}$, where $n$ is a positive integer.
								 Observe now that the projection $p_m$ of $I^{\omega}$ on the first $m$ coordinates satisfies
								  $p_m(\overline{K(F, \frac{1}{n})}) = \overline{K^{m}(p_m(F), \frac{1}{n})}$, where $\overline{K^{m}(p_m(F), \frac{1}{n})}$
									is a ball at $p_m(F)$ included in the $m$-dimensional Euclidean cube $I^{m}$. Moreover, we have the following:
									
					$$ (23) \ \ \ \ \ \ \overline{K(F, \frac{1}{n})} = \bigcap_{m=1}^{\infty} p^{-1}_{m}(\overline{K^{m}(p_m(F), \frac{1}{n})}).$$
					
					 Now, we observe that for every positive integer $n$, any closed\\
                     $n$-dimensional cube $Q$ is a regular compact metric space.
					 Indeed, since $\lambda (K) \leq \mu(K)$ for any closed subset $K$ of $Q$, where $\lambda$ is the Banach-Mycielski content
					and $\mu$ is the Lebesgue measure (which is just Banach-Mycielski measure, see [10], Thm. 1) 
					 defined in $Q$, then the positive singularity of $K$ would imply that $\mu$ is not outer regular.
					  Therefore, by $(23)$, we get that $(I^{\omega}, d_a)$ is regular. Then we can apply Theorem 5.2 to get a countably additive
						congruence-invariant Borel measure in $I^{\omega}$. Now, it is enough to apply the Kolmogoroff extension theorem to get that
						 $m$ is the Lebesgue product measure. $\Box$

\author{Grzegorz Tomkowicz\\
       $\quad$ Centrum Edukacji $G^2$\\
        ul.Moniuszki 9\\
        41-902 Bytom\\
        Poland\\
      e-mail: gtomko@vp.pl}

\end{document}